\documentclass[12pt, oneside]{amsart}   	
\usepackage{geometry}                		
\usepackage{graphicx}				
\usepackage{amssymb}
\usepackage{url}	
\newenvironment{dedication}
  {
   \thispagestyle{empty}
   \vspace*{\stretch{1}}
   \itshape             
   \raggedleft          
  }
  {\par 
   \vspace{\stretch{2}} 
  }

\theoremstyle{plain} 
\newtheorem{theorem}{Theorem}[section]

\newtheorem{corollary}[theorem]{Corollary}

\theoremstyle{definition} 

\newcommand{\Z}{\mathbb{Z}}

\newcommand{\C}{\mathbb{C}}

\newcommand{\re}{\operatorname{Re}}

\title{On Alan Schoen's I-WP Minimal Surface}
\author{Dami Lee and Matthias Weber and A. Tom Yerger}
\date{}							

\begin{document}
\maketitle

\begin{dedication}
Dedicated to the memory of Alan Schoen (December 11, 1924 – July 26, 2023)
\end{dedication}

\begin{figure}[h] 
   \centering
   \includegraphics[width=4in]{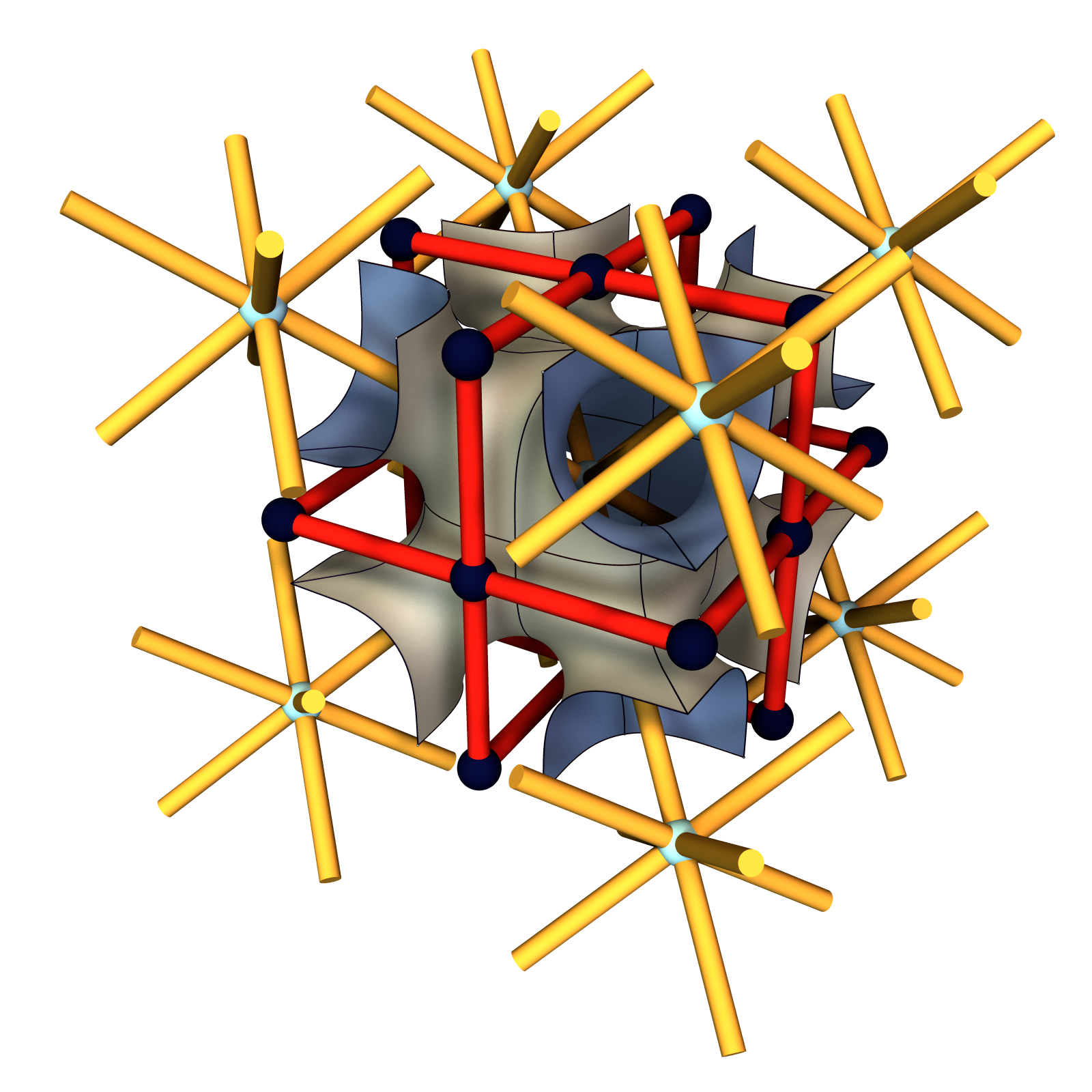} 
    \caption{Schoen's I-WP surface with its skeletal graphs}
   \label{fig:skeleton}

 \end{figure}

\begin{abstract}
We discuss in detail Alan Schoen's I-WP surface, an embedded triply periodic minimal surface of genus 4 with cubical symmetries.
We exhibit various geometric realizations of this surface with the same conformal structure and use them to prove
that the associate family of the I-WP surface contains six surfaces congruent to I-WP at Bonnet angles that are multiples of $60^\circ$.
\end{abstract}

\section{Introduction}

In his 1970 NASA report (\cite{sch1}), Alan Schoen describes 17 triply periodic embedded  minimal surfaces, most of them new. His construction was based on finding  highly symmetric skeletal graphs for their complements. One of these surfaces, the gyroid, became immensely famous because it  unexpectedly lies in the associate family of the P- and D-surface discovered by Hermann Amandus Schwarz (\cite{schw1}).
For every minimal surface there is an associate family of locally isometric minimal surfaces with identical Gauss map, parametrized by the Bonnet angle $\theta$. Surfaces in this family whose Bonnet angles differ by $\pi$ are congruent, and surfaces corresponding to the Bonnet angles  $\theta=\pm\pi/2$ are called conjugate. For instance, the P- and D-surfaces are conjugate. That there are further embedded minimal surfaces in the associate family is exceedingly rare.

Alan Schoen's gyroid was the first exception.
Later Fogden, Haeberlein and  Lidin found another  embedded  surface in the associate family of a particular Schwarz H-surface (\cite{fhl1}). Both these surfaces were  shown to be embedded in 1996 by Karsten Große-Brauckmann and Meinhard Wohlgemuth (\cite{gbw1}).

Another surface of particular interest in Alan Schoen's list is his I-WP surface. Its quotient by its translational lattice  has comparatively low genus 4, and it has cubical symmetries in addition to its translational symmetries. Alan Schoen mentions that the I-WP surface is in fact the conjugate of a surface that had been discovered already in 1934 by Berthold Steßmann (\cite{stess34}).

In their paper \cite{Lidin-IWP} 
Lidin, Hyde and Ninham give an explicit Weierstrass representation of the I-WP surface and numerically examine its associate family. They find five more congruent copies of I-WP for associate angles which  they numerically determine to be  multiples of $\pm2\pi/6$ ``up to seven decimal places''.

We will provide a simple explanation and proof  of this observation by exhibiting a hidden conformal symmetry of the I-WP surface: The order 4 rotation about any coordinate axis that leave the surface invariant are in fact third powers of an order 12 intrinsic conformal automorphism that is not visible in Euclidean space.

To this end, we will identify several different conformal representations of I-WP, using a 12-fold cyclically branched  cover over a thrice punctured sphere, a triple cover over the sphere that is branched at the vertices of a regular octahedron, and a Euclidean  approximation of I-WP.

We organize this paper as follows:

\begin{enumerate}
\item 
Section 2 outlines the history of the Steßmann surface (the conjugate of I-WP). 
\item In section 3 we represent I-WP as a 3-fold branched cover over the octahedron. 
\item
In section 4, we construct a specific 12-fold branched cover over the thrice punctured sphere and a triply periodic polyhedral surface, both conformally equivalent to I-WP.
\item Section 5 uses translation structures to find a basis of holomorphic 1-forms of I-WP, find its Weierstrass representation, and prove our main result about its associate family. 
\end{enumerate}

\section{Steßmann's Surface}

In his 1867 monograph \cite{schw1} {\em Bestimmung einer speciellen Minimalfläche,} Hermann Amandus Schwarz gives equations in terms of elliptic integrals for   several triply periodic minimal surfaces that solve the Plateau problem for certain spatial quadrilaterals. These quadrilaterals have the property that rotations about their edges generate a discrete group.

In 1891, Arthur Moritz Schoenflies  showed that there are six  such quadrilaterals (up to similarity) but didn't provide equations for the Plateau solutions.

\begin{figure}[h] 
   \centering
   \includegraphics[width=2.5in]{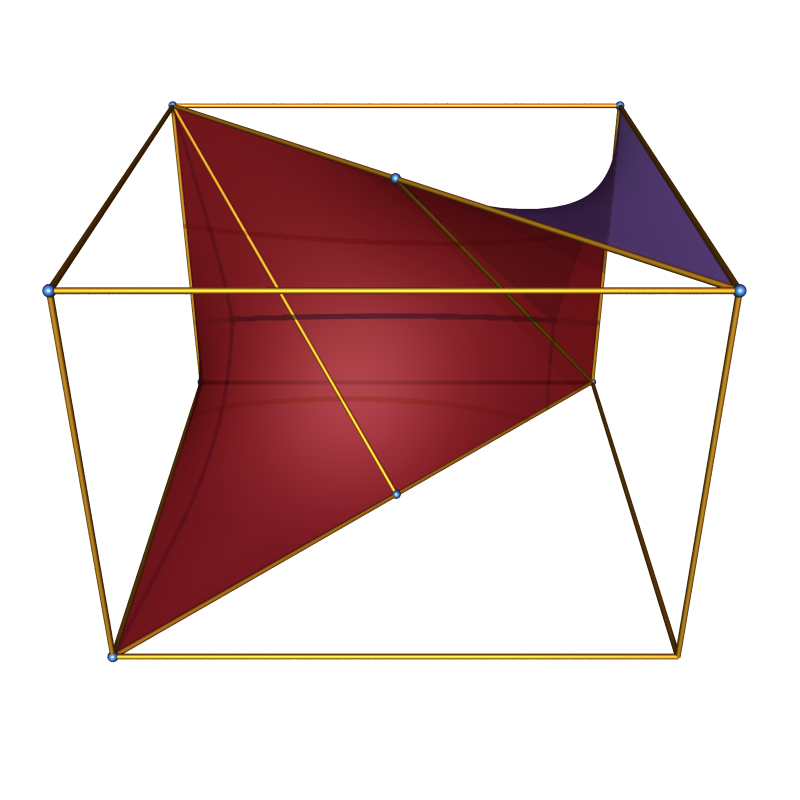} 
   \includegraphics[width=2.5in]{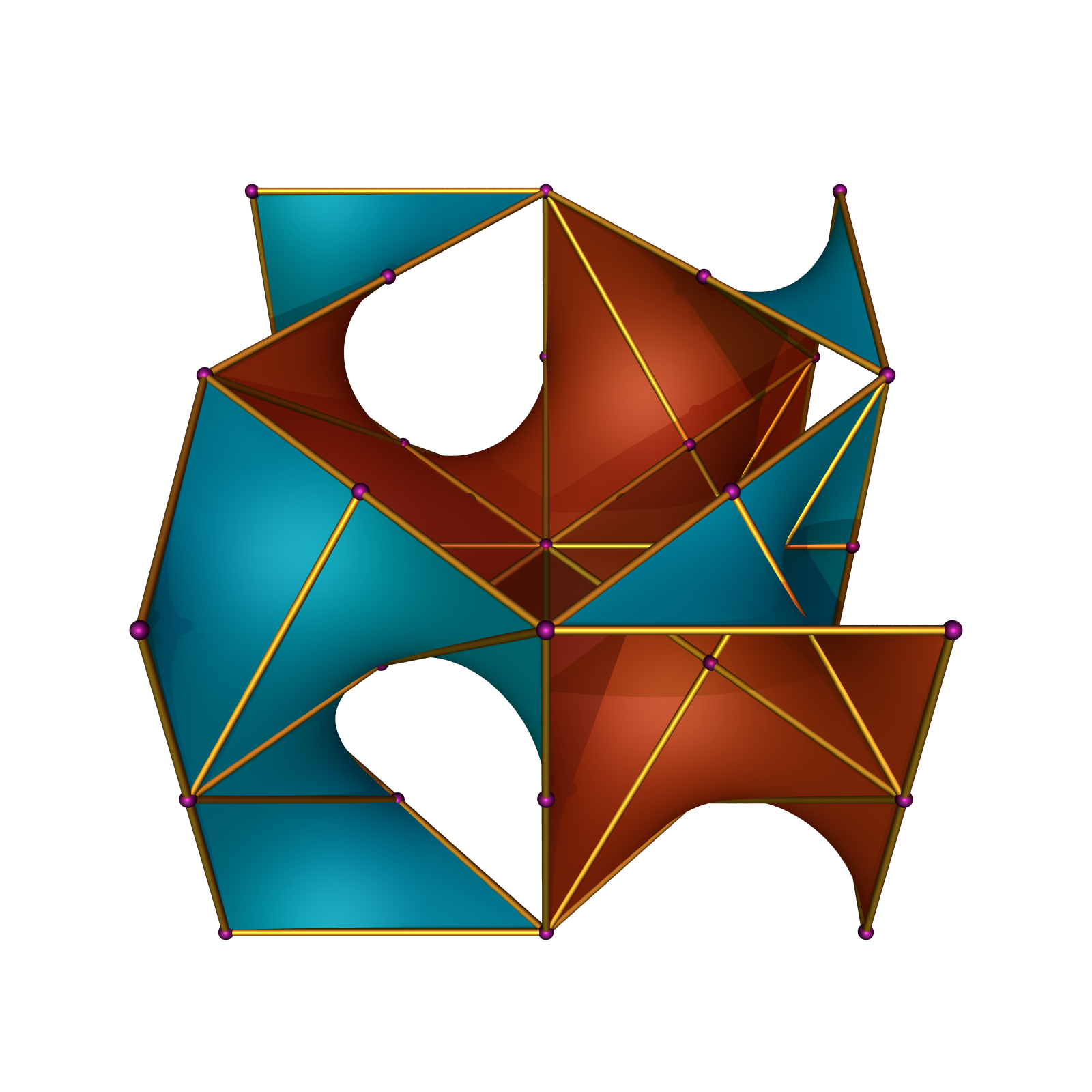} 
   \caption{Steßmann's Surface}
   \label{fig:stessmann}
\end{figure}

This had to wait until 1934 when  Berthold Steßmann (\cite{stess34}), a student of Carl Ludwig Siegel, completed this for the remaining cases in the spirit of Schwarz. One of them (case VI in his list) he treats  in full detail.

Its contour is shown in figure \ref{fig:stessmann}. The left image shows  three copies of this contour that fit together to form a spatial hexagon which can be described as follows: Take a rectangular box of dimension $1\times 1\times \sqrt2/2$. Then the contour above consists of two (non-parallel) diagonals of the top and bottom face, two vertical edges of the box, and two horizontal edges that lie diametrically across.

Now add the two line segments that connect a midpoint of a diagonal to a vertex of the box as shown in the figure. These two line segments divide the hexagon into three quadrilaterals, the Schoenflies quadrilaterals.

The dimensions of the box are chosen such that these quadrilaterals are congruent by $180^\circ$ rotations about the newly added lines.

\begin{figure}[h] 
   \centering
   \includegraphics[width=3.5in]{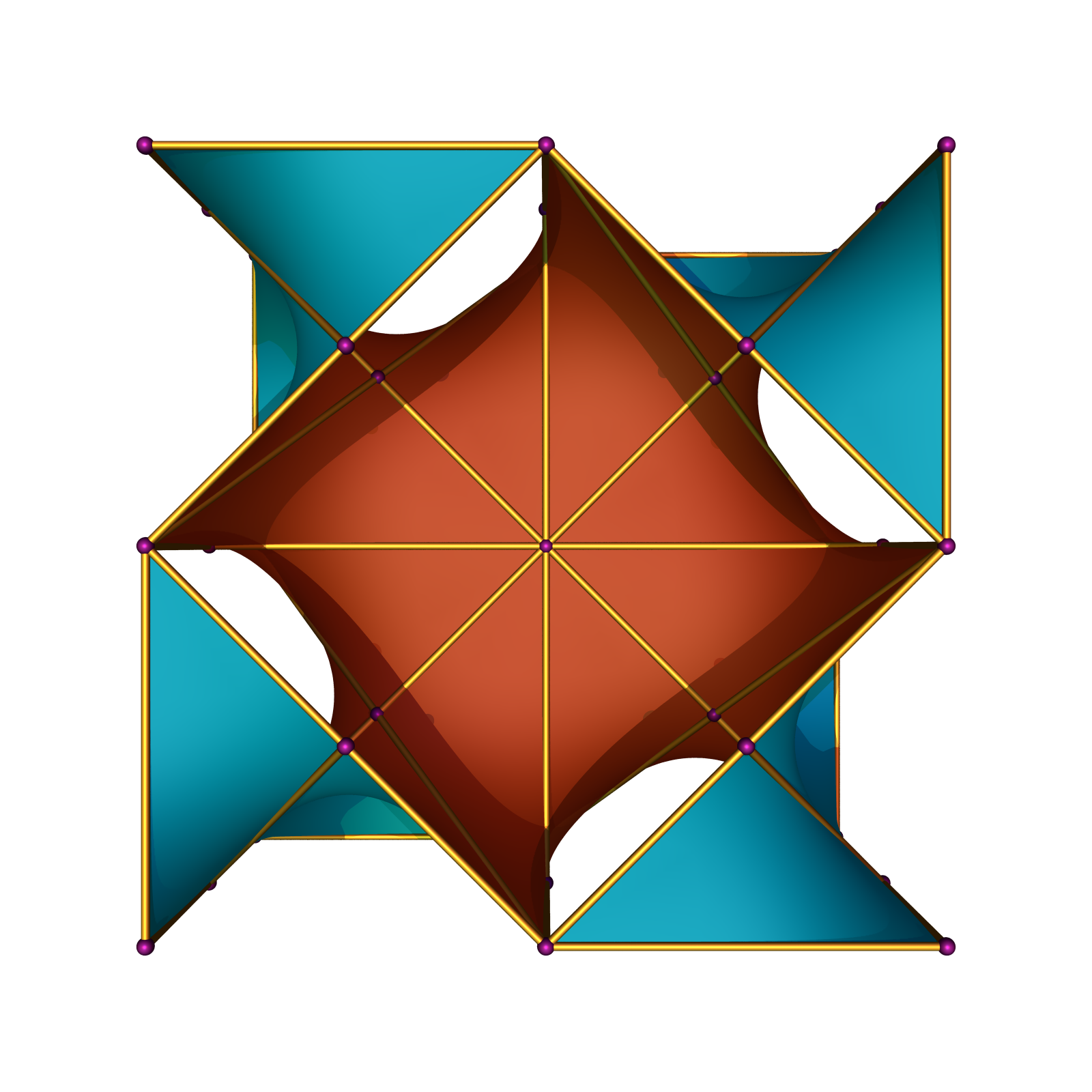} 
   \caption{Steßmann's Surface --- top view}
   \label{fig:stessmann2}
\end{figure}

Extending the surface further produces the appealing triply periodic surface of figure \ref{fig:stessmann2}. Unfortunately, this surface will not stay embedded; you see this at the corners where three pairwise orthogonal edges meet. 

The internet knows little about Berthold Steßmann. There is a short biographical note in \cite{tobies}, p. 330, telling us that he was born on August 4, 1906 in Hallenberg, Germany. He completed his studies in Göttingen and Frankfurt to become a high school teacher in 1933. Then, a year later, he received his PhD about periodic minimal surfaces, with Carl Ludwig Siegel as advisor. The same year,  Mathematische Zeitschrift published a paper of Steßmann, covering the same topic. He emigrated in 1935 to Palestine (\cite{glade}, p. 236).

\section{Schoen's I-WP Surface} 

\begin{figure}[h] 
   \centering
   \includegraphics[width=2.5in]{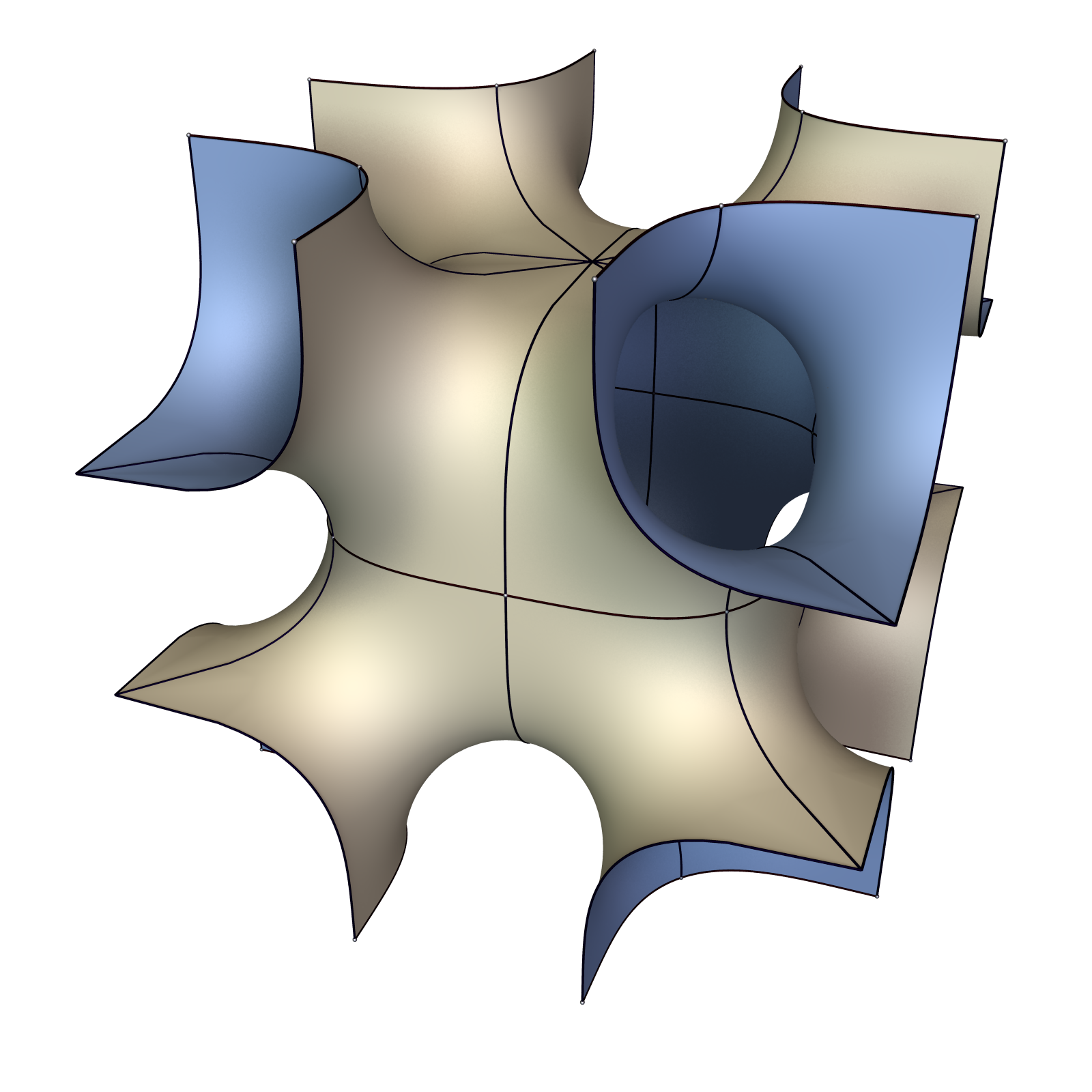} 
   \includegraphics[width=2.5in]{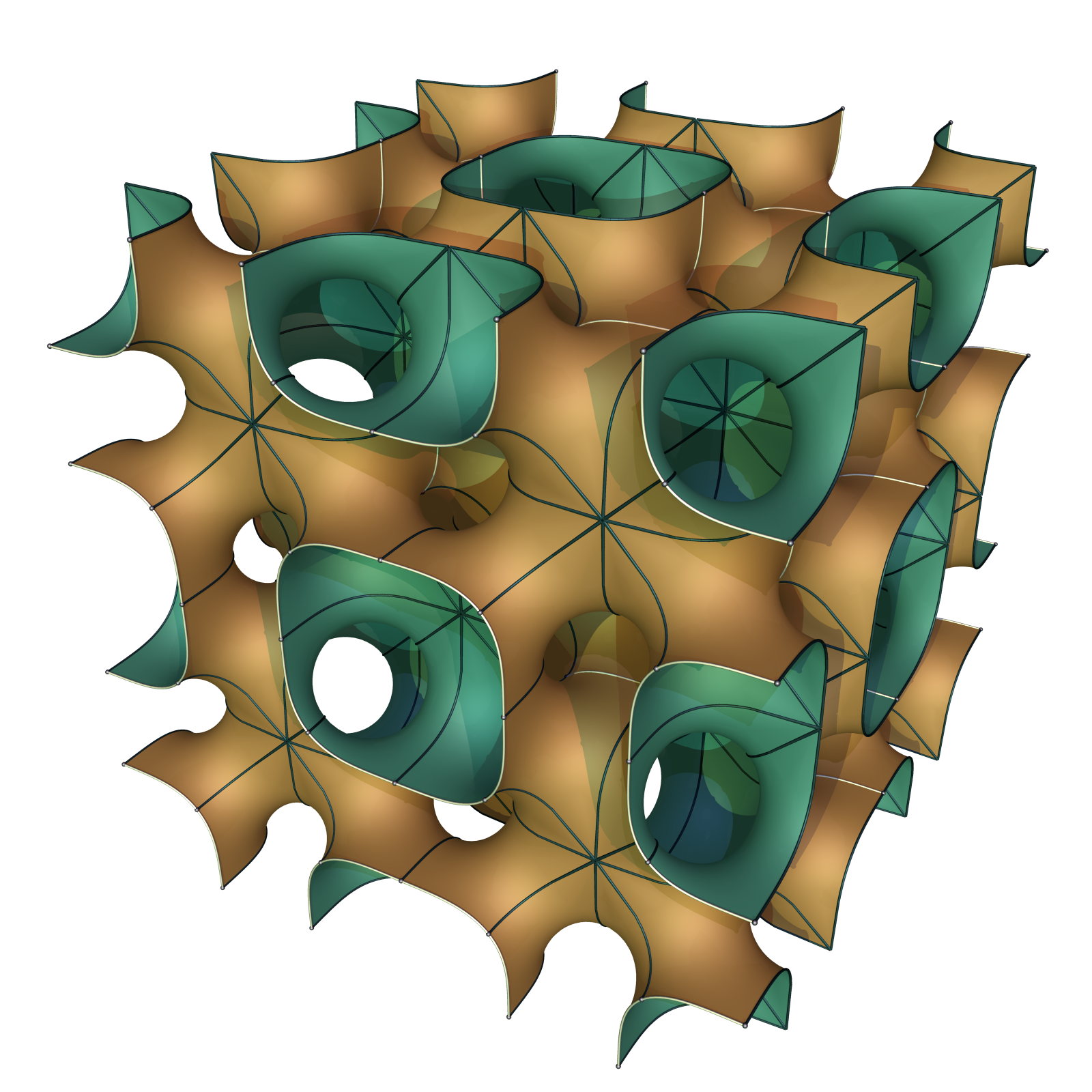} 
   \caption{The I-WP Surface}
   \label{fig:IWP}
\end{figure}

Surprisingly, the conjugate minimal surface of Steßmann's surface is embedded:
Alan Schoen’s I-WP surface from 1970 \cite{sch1} is a triply periodic minimal surface of genus 4. One can think of it as a sphere extending handles towards the vertices of a cube, see figure \ref{fig:IWP}.

The name ``I-WP" describes the skeletal graph of the two complements of the surface, see figure \ref{fig:skeleton}. The ``I" stands for the body-centered I-Bravais lattice, 
while ``WP" stands for {\em wrapped package}, indicating the appearance of the skeletal graph of the surface near a vertex, where four sticks are bundled together.

After identifying opposite faces of a cubical cell (see the left image of figure \ref{fig:IWP}) the surface has genus 7, but half of such a cell already constitutes a translational fundamental domain under the body-centered I-Bravais lattice, bringing the genus down to 4.

A first detailed description of the Weierstrass data appears in \cite{ka5}. Other descriptions of I-WP in terms of Weierstrass data have been given by  \cite{Lidin-IWP}  and \cite{fuwe1}.

Let's denote the underlying Riemann surface of the quotient of I-WP by its translational symmetries by $S$. Recall that the stereographic projection of the Gauss map $G$ is a meromorphic function on $S$.

\begin{theorem}
For $S$, the Gauss map $G$ is a meromorphic function of degree $3$ and represents $S$ as a 3-fold  cover of the sphere branched over the vertices of a regular octahedron.
\end{theorem}
\begin{proof}
For a triply periodic minimal surface  of genus $g$, the Gauss map $G$ is a meromorphic function of degree $g-1$ (see \cite{me4}, Theorem 3.1), so for $S$, $G$ has degree 3. 

A $90^\circ$ rotation about a coordinate axis is an automorphism of $S$ of degree 4 with two fixed points which are branch points of index 3 each. This gives us six branch points of $G$ with total index 18 and branch values corresponding to the vertices of an octahedron. By the Riemann-Hurwitz formula, any genus 4 Riemann surface that is a threefold branched cover over a sphere has total branching index 18. Therefore $G$ is otherwise unbranched and  realizes $S$ as a 3-fold cover  branched over the vertices of a regular octahedron.
\end{proof}

The embedding of I-WP in Euclidean space gives no  obvious reason why this covering could  also be the quotient map corresponding to an order 3 automorphism. In fact, the only order 3 automorphisms of the cubical symmetry group are rotations about the space diagonals of the cube, and these act without fixed points on I-WP.

The preimages of the faces of the octahedron become a 12-valent map on I-WP consisting of 24 triangles. We will see in the next section that this map is considerably more regular than expected.

\section{$S_{147}$: A 12-fold branched cover of the thrice punctured sphere}

In this section we will show that I-WP is conformally equivalent to a cyclic 12-fold branched cover $S_{147}$ over the thrice punctured sphere. The reason for this surprising fact is that the order 4 rotations about the coordinate axes are in fact third powers of conformal automorphisms. We begin by giving a concrete construction of this branched cover.

Take a thrice punctured sphere with punctures at points $p_1$, $p_2$, $p_3$. Pick an arbitrary fourth point $p_0$ and make slits from $p_0$ to $p_j$. Denote this slit sphere by $\Sigma$. Now take 12 copies of $\Sigma$ and denote them by $\Sigma_i$, $i=0,\ldots, 11$. Let $d_1=1$, $d_2=4$, $d_3=7$. Identify the left hand side  of the slit from $p_0$ to $p_j$ (as seen from $p_0$) in copy $\Sigma_i$ with the right hand side of the same slit in copy $\Sigma_k$ where $k\equiv i +d_j \pmod {12}$. In other words: Think of the 12 copies of $\Sigma$ being stacked as a tower. When one crosses a slit $p_0p_j$ from the left to the right, one reemerges $d_j$ stories above.

The result is a Riemann surface $S_{147}$, the indices referring to the values of $d_j$.

\begin{theorem}
$S_{147}$ is a compact Riemann surface of genus 4, given by the equation $y^{12} = (x-p_1)^1(x-p_2)^4(x-p_3)^7$. 
\end{theorem}
\begin{proof}
 The equation defines a Riemann surface with the same branching behavior as described. Its genus can be computed using Euler's formula by representing the sphere as a doubled triangle. Its three vertices, three edges and two faces lift to a triangle map of $S_{147}$ consisting of 24 triangles, 36 edges and $1+ 4 +1=6$ vertices coming from the preimages of the branch points $p_1$, $p_2$ and $p_3$ respectively. This implies that the Euler characteristic of $S_{147}$ is $-6$ and its genus  is 4.
\end{proof}

Because a thrice punctured sphere has only one conformal structure, it is easy to find other conformally correct geometric realizations of $S_{147}$. 
We begin with its hyperbolic structure which we will  identify with a polyhedral surface in Theorem \ref{thm:ident}.

To this end, we double a hyperbolic triangle with angles $\pi/12, \pi/6,  \pi/12$ to obtain a thrice punctured sphere with a hyperbolic structure.  We employ this sphere for the branched covering construction above. The choice of the  angles implies that the resulting surface carries a smooth hyperbolic structure, because at each puncture, the angles of all incident triangles add up to $360^\circ$. 
Developing this hyperbolic structure on $S_{147}$  into the hyperbolic plane results in a 24-gon, representing a fundamental domain for $S_{147}$, see figure \ref{fig:124}. We have placed the vertex $p_1$ at the center of the disk model. In  this figure the edge identifications are given by numbers, and a closed geodesic is indicated in blue.
 
\begin{figure}[h] 
   \centering
   \includegraphics[width=3in]{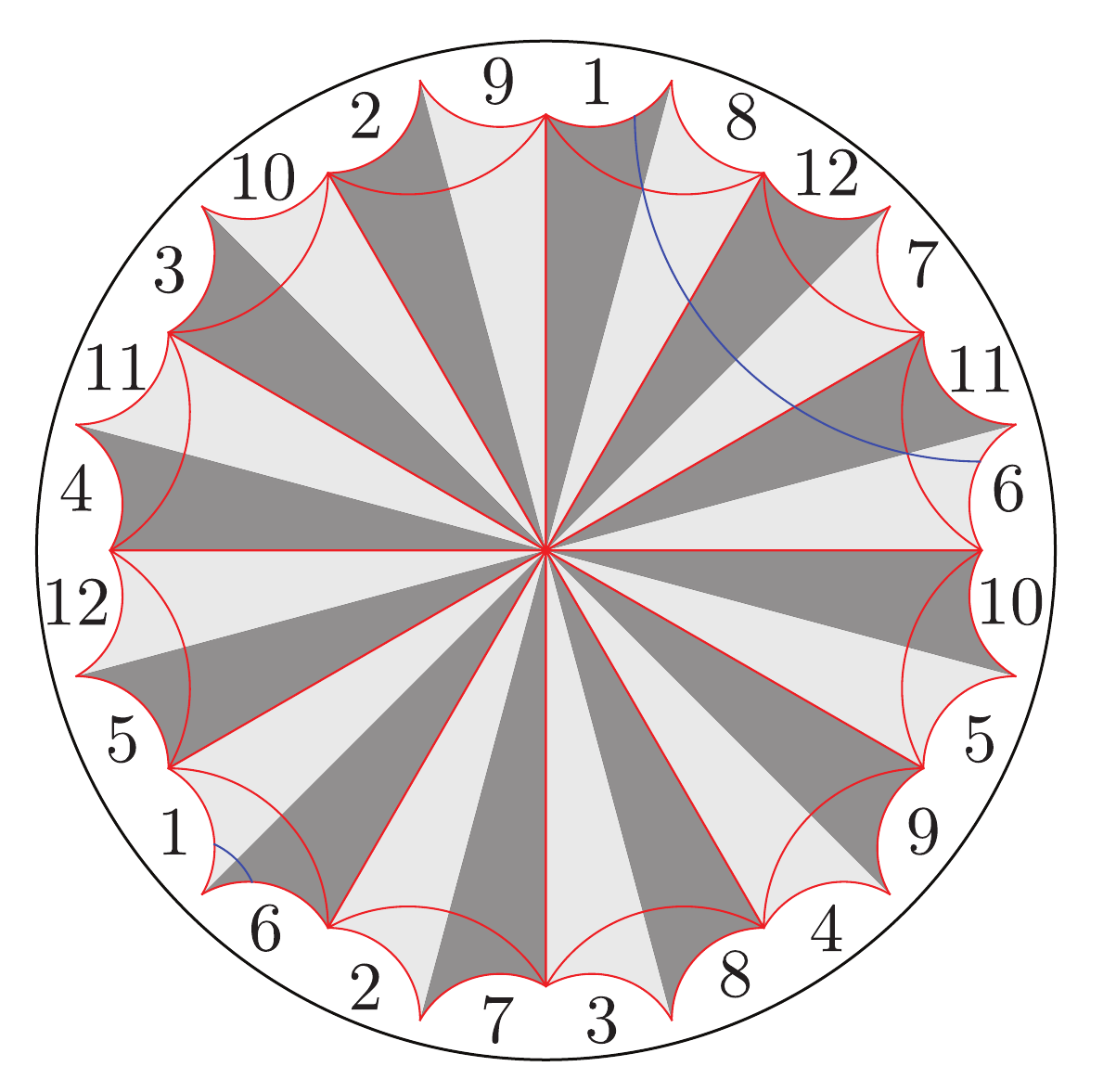} 
   \caption{Hyperbolic Fundamental Domain of $S_{147}$}
   \label{fig:124}
\end{figure}

Given the construction, $S_{147}$ has an automorphism of order 12, visible in figure \ref{fig:124} as a $30^circ$ rotation about the center of the disk.
We will now show that it is in fact much more symmetric.

A second hyperbolic triangle tiling of $S_{147}$ can be obtained as follows: Take two triangles that are adjacent along the edge connecting the two $\pi/12$ vertices and subdivide this quadrilateral along the diagonal connecting the two $\pi/6$ vertices into two equilateral $\pi/12$-triangles. These triangles are indicated in figure \ref{fig:124} as red. 
We will denote $S_{147}$ together with this new tiling by $S'_{147}$.
We will show that this tiling is in fact Platonic. To this end, we will digress and introduce a polyhedral surface that is combinatorially and conformally equivalent to $S'_{147}$.

In {\em The Symmetries of Things}  (\cite{cbg}, page 342) by John Horton Conway, Heidi Burgiel and Chaim Goodman-Strauss there are a description and an image of a pseudo-Platonic polyhedron called {\em Octahedral $3^{12}$}, which we will denote by $\tilde S_{3^{12}}$. 

One way to construct this polyhedron is by placing an octahedral annulus (in other words a regular antiprism over a triangle without the top and bottom triangles)  inside a cube so that its vertices lie on the edges of the cube. 
The precise dimensions of the octahedral annulus are chosen so that its vertices (represented by yellow dots) divide the cube edges in the proportion 1:3 which guarantees that the octahedron becomes regular.
Replicating it by reflecting at the faces of the cube generates a triply periodic polyhedron $\tilde S_{3^{12}}$. On the left of figure \ref{fig:312} we show 8 copies of the octahedral annulus inside a $2\times2\times2$ cube. This constitutes twice a translational fundamental domain, the translations given by the diagonal vectors of the cube. 

The (translational) compact quotient surface $ S_{3^{12}}$ is tiled by $f=24$ equilateral triangles with 12 of them meeting at each vertex. This shows it has $v=6$ vertices and $e=36$ edges so that by Euler's theorem its Euler characteristic is $-6$ and the genus is 4.

\begin{figure}[h] 
   \centering
   \includegraphics[width=2.5in]{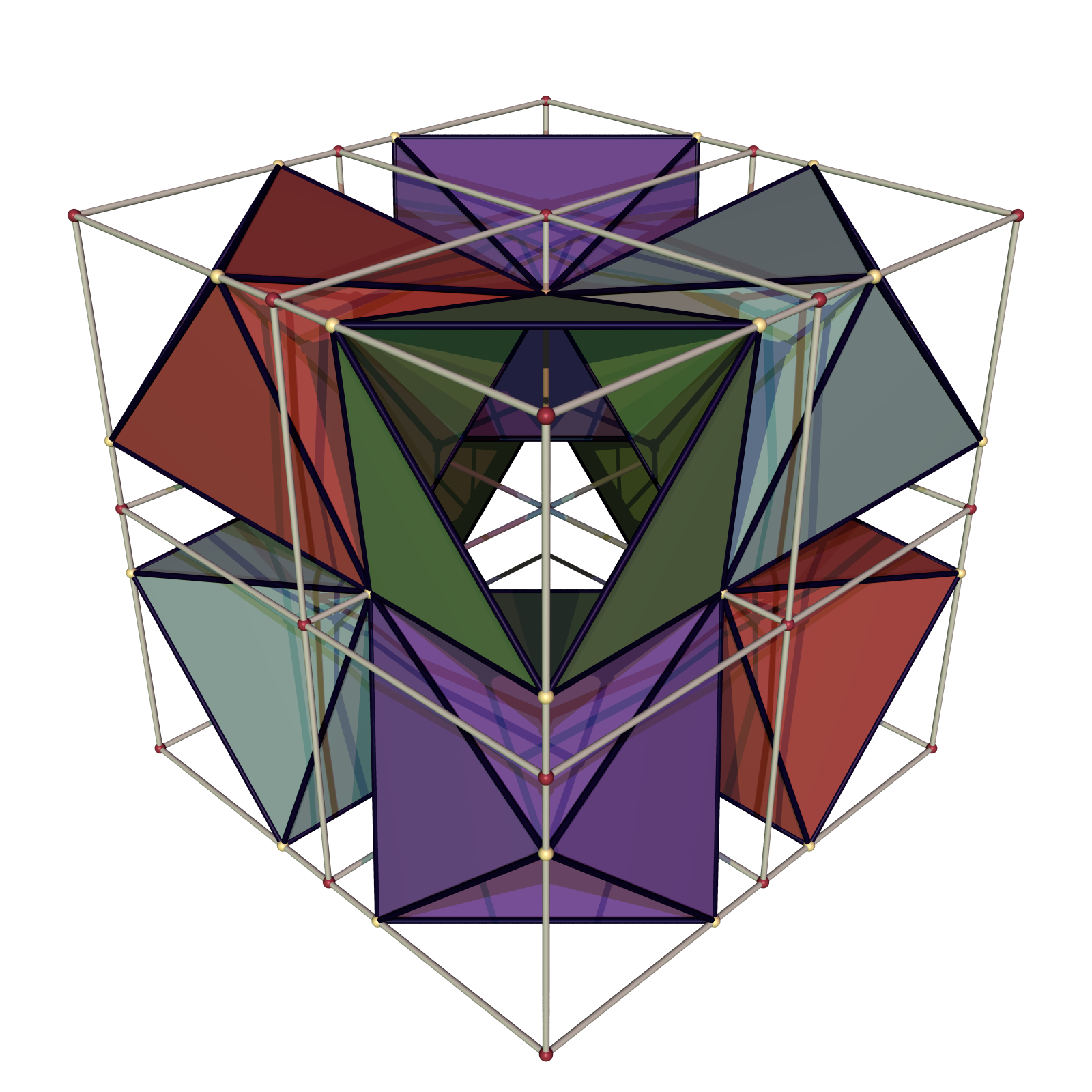} 
   \includegraphics[width=2.5in]{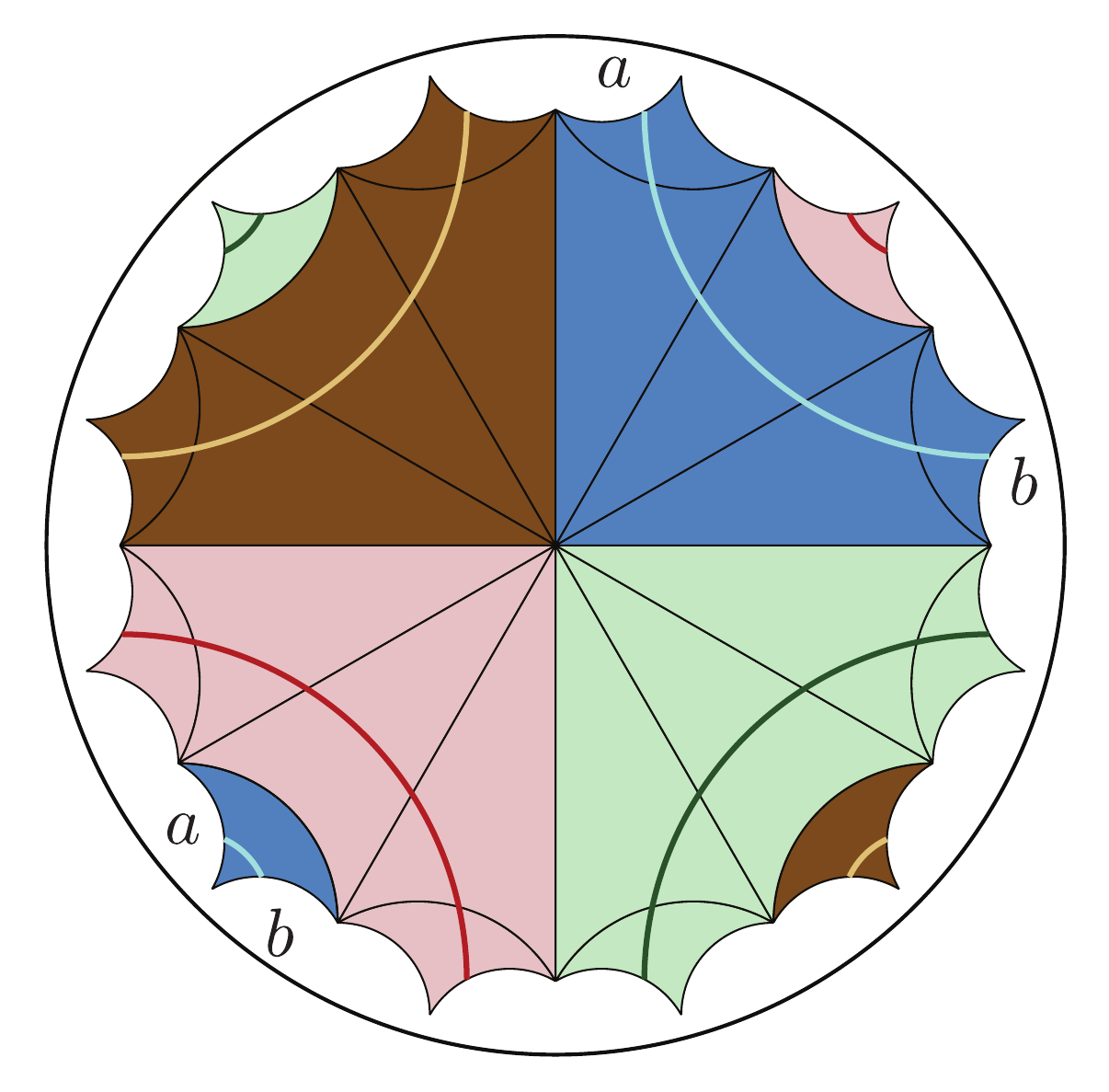} 
   \caption{The $\tilde S_{3^{12}}$ periodic polyhedral surface and its hyperbolic fundamental domain}
   \label{fig:312}
\end{figure}

Recall that any polyhedron carries a natural conformal structure: To see this, note that a neighborhood of a vertex is intrinsically isometric to an abstract Euclidean cone with a  cone angle given by the sum of the angles of the adjacent polygons (and possibly larger than 2$\pi$). Such a cone can be obtained by identifying the two edge rays of a sector  of angle $\alpha$ in the universal cover of the  punctured complex plane $\C-\{0\}$. A conformal chart for this cone is given by the map $z\mapsto z^{2\pi/\alpha}$.

\begin{theorem}\label{thm:combi} (\cite{damithesis}) The triangle maps of $S_{3^{12}}$ and $S'_{147}$ are isomorphic maps.
\end{theorem}
\begin{proof}
We  develop  $S_{3^{12}}$  into the hyperbolic plane by  (abstractly)  replacing each Euclidean triangle with a hyperbolic $30^\circ$ triangle as shown in the right figure of  \ref{fig:312}. To make this easier to follow, we have colored the triangles corresponding to four individual anti-prisms  with the same color. The  waist polygons of the anti-prisms become the shaded geodesics. That they close on each antiprism after passing through six triangles enforces the identification pattern of this 24-gon, indicated by letters. We see that the identification patterns of this new 24-gon and the one from figure \ref{fig:124} are the same. Thus both triangle maps are combinatorially identical.
\end{proof}

As a corollary, we prove:

\begin{theorem}(\cite{tomthesis})
The triangle maps of $S_{3^{12}}$ and  $S'_{147}$ are Platonic. The group of its orientation preserving automorphisms is of order 72, and the full automorphism group has order 144.
\end{theorem}

\begin{proof}
Using the cubical symmetries of $S_{3^{12}}$, we find 24 orientation preserving automorphisms of $S_{3^{12}}$. A rotation about a suitable coordinate axis becomes a $90^\circ$ rotation about the center of the developed version of $S_{3^{12}}$.
In $S'_{147}$, this rotation is the third power of a $30^\circ$ rotation.
  Thus we have at least 72 orientation preserving automorphisms that preserve both triangle maps. On the other hand, the (orientation preserving) automorphism group of any surface that preserves a triangle map with 24 triangles can at most be of order 72. It follows that the maps are Platonic, as we can map any pair of (vertex, adjacent edge) to any other such pair. 
  
Either map has reflectional symmetries that fix edges, so the order of the full automorphism group is at least 144. Again, any (full) automorphism group of a surface that preserves a triangle map with 24 triangles can at most be of order 144.

This proves all our claims.

\end{proof}
 
\begin{theorem}\label{thm:ident} (\cite{damithesis}) $S_{3^{12}}$ and $S'_{147}$ are conformally equivalent under a map that preserves the respective triangle maps.
\end{theorem}

\begin{proof}
We have shown that the Euclidean triangle map of $S_{3^{12}}$ and the hyperbolic triangle map of  $S'_{147}$ are combinatorially equivalent and Platonic. We define a map from $S_{3^{12}}$ to $S'_{147}$ by taking the conformal map from a regular Euclidean triangle to a  hyperbolic $\pi/12$ triangle that is equivariant with respect to the $120^\circ$ rotations of the triangles. We apply this map to any pair of triangles of $S_{3^{12}}$ and $S'_{147}$.  As both maps are Platonic, we can extend this map (at first locally) by using the Schwarz reflection principle. As both maps are combinatorially equivalent, the extension of this map becomes a well defined conformal map from $S_{3^{12}}$ to $S'_{147}$. That this map is also conformal at the vertices follows because both maps have  conformal coordinates at their vertices, using the  Riemann singularity theorem.
\end{proof}

The fact that we have an order 12 rotation about any vertex of $S_{3^{12}}$ allows us to study other quotients. Of particular importance is the following, because it provides us with a crucial link to the Gauss map of the I-WP surface:

\begin{corollary}
The $120^\circ$ rotation about a vertex of the platonic triangle tiling of $S_{3^{12}}$ has a regular octahedron as its quotient.
\end{corollary}
\begin{proof}
By inspection of the hyperbolic fundamental domain.
\end{proof}

Finally we are ready to prove:

\begin{theorem}
Consider I-WP together with the triangle map coming from the covering over the octahedron by its Gauss map. Then I-WP is conformally equivalent to both $S_{3^{12}}$ and $S_{147}$ under a map that preserves the respective triangle maps.
\end{theorem}
\begin{proof}
We have seen that both I-WP and $S_{3^{12}}$ are branched covers over an octahedron, the former via its Gauss map, the latter  via the quotient map  under an order 3 rotation. It remains to show that both coverings are topologically equivalent.

The left of figure \ref{fig:dodecagon} shows a portion of I-WP corresponding to the 12 triangles that are the preimage of the upper hemisphere under the Gauss map. At the center, the Gauss map is branched with value $\infty$, and  the 12 vertices are cyclically mapped  to $1, -i, -1, i$.

 \begin{figure}[h] 
   \centering
   \includegraphics[width=2.5in]{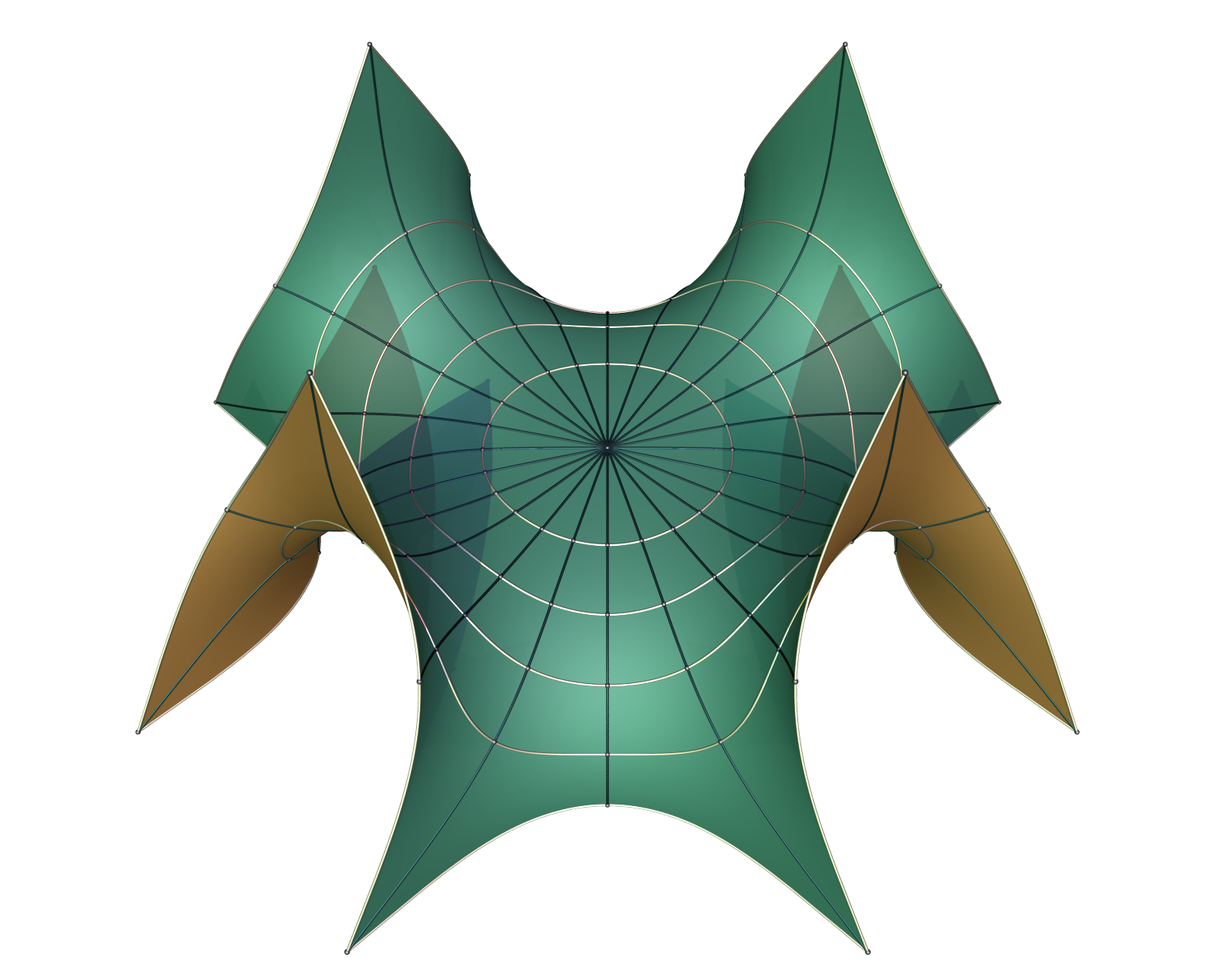} 
   \includegraphics[width=2.5in]{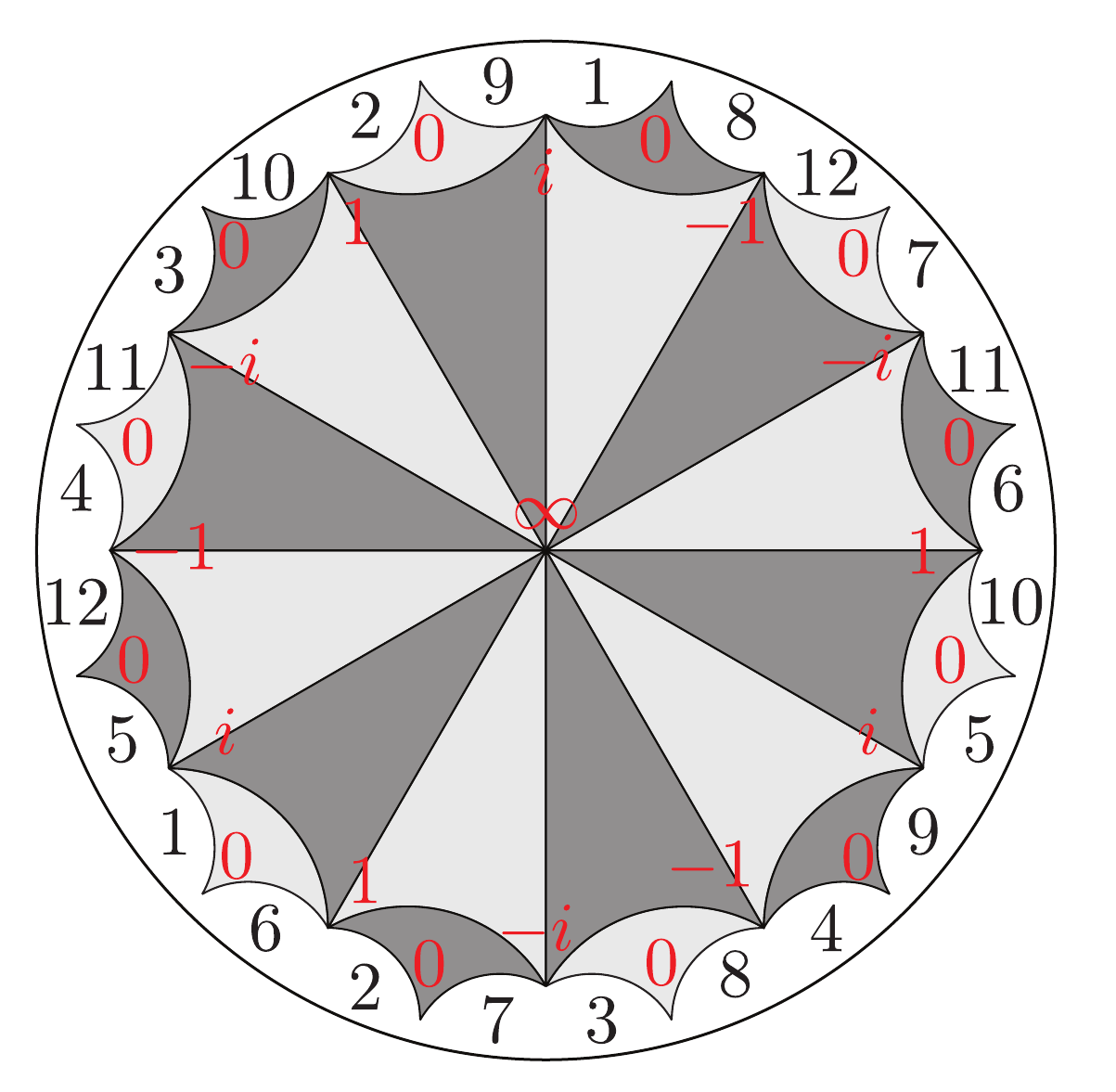} 
   \caption{Minimal Dodecagon}
   \label{fig:dodecagon}
\end{figure}

This allows us to assign Gauss map values to the vertices of the triangle map of $S_{3^{12}}$ (the right image in figure \ref{fig:dodecagon})  that are consistent with the identification pattern of $S_{3^{12}}$, thus showing the topological equivalence of both coverings.

\end{proof}

\section{Weierstrass Representation}

To write down a Weierstrass representation for I-WP, we will use $x$ as a coordinate on $S_{147}$, using $p_1=0$, $p_2=1$, $p_3=\infty$ with the defining equation $y^{12}=x(x-1)^4$.

\begin{theorem}
The 1-forms
\begin{alignat*}{2}
   \omega_1 = {}& x^{1/12-1}(x-1)^{4/12-1}\, dx =& \frac{y}{x(x-1)}\, dx\\
   \omega_2 = {}& x^{2/12-1}(x-1)^{8/12-1} \, dx=& \frac{y^2}{x(x-1)}\, dx\\
   \omega_3 = {}& x^{4/12-1}(x-1)^{4/12-1}\, dx =& \frac{y^4}{x(x-1)^2}\, dx\\
   \omega_4 = {}& x^{7/12-1}(x-1)^{4/12-1} \, dx=& \frac{y^7}{x(x-1)^3}\,dx\\
\end{alignat*}
form a basis of holomorphic 1-forms on $S$. 
\end{theorem}
\begin{proof}
This can of course be checked using suitable local coordinates. Instead we will make this evident by geometrically constructing the translation structures of these 1-forms.

Let's replace the hyperbolic $(1, 2, 1)\frac\pi{12}$-triangle of $S_{147}$ conformally by a Euclidean $(1,4,7)\frac\pi{12}$ triangle in the hyperbolic fundamental domain. The resulting 24-gon is shown in figure \ref{fig:trans124}.

\begin{figure}[h] 
   \centering
   \includegraphics[width=3in]{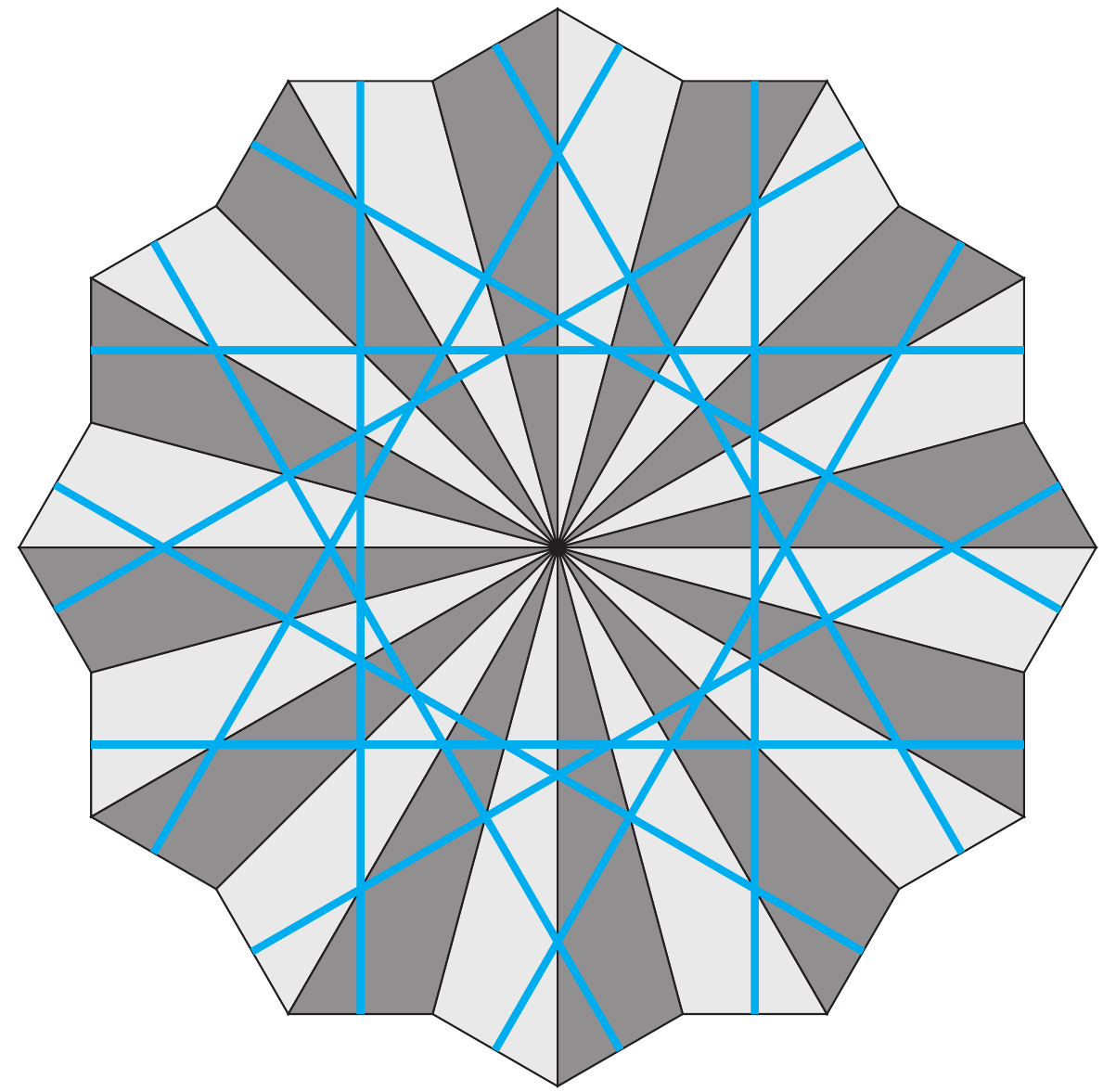} 
   \caption{Translation structure for $\omega_1$}
   \label{fig:trans124}
\end{figure}

We see that the edge identifications can be realized as Euclidean translations. Thus the identification space defines a translation structure on $S_{147}$ with cone points at the preimages of the $p_i$.

The same is true if we replace the hyperbolic $(1, 2, 1)\frac\pi{12}$ triangle by a Euclidean triangle with angles $(2,8,2)\frac\pi{12}$, $(4,4,4)\frac\pi{12}$, $(7,4,1)\frac\pi{12}$, respectively. Below is a representation of the translation structure corresponding to $\omega_2$. The two hexagons are identified using the two red vertical slits, creating an interior branch point with cone angle $4\pi$.

\begin{figure}[h] 
   \centering
   \includegraphics[width=4in]{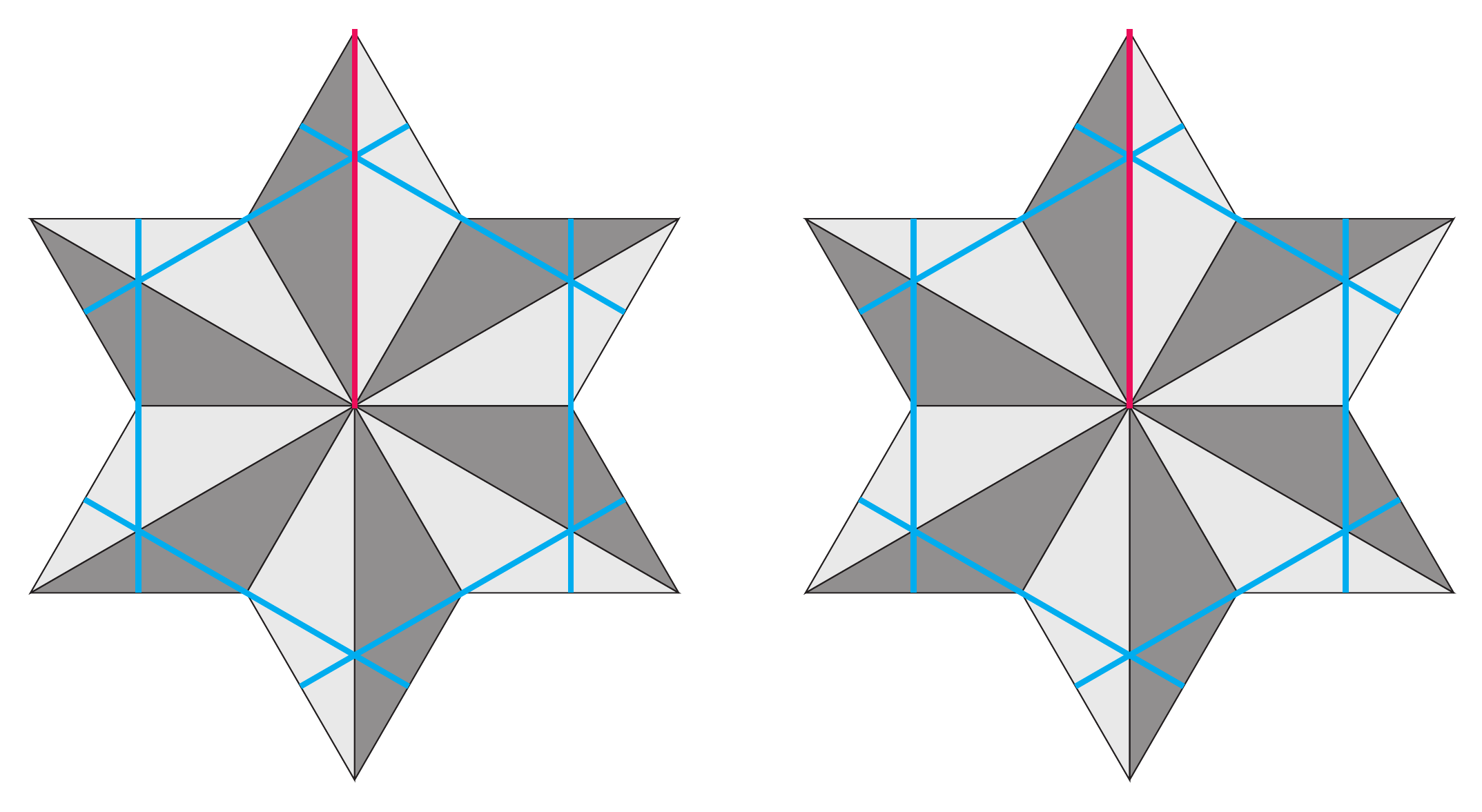} 
   \caption{Translation structure for $\omega_2$}
   \label{fig:124b}
\end{figure}

From the cone points of these translation structures we can determine the divisors of the corresponding forms: A cone angle of $2\pi k$ corresponds to a zero of order $k-1$. Thus (noting that $p_2$ has four preimages denoted by $p_{2,j}$)

\begin{align*}
   (\omega_1) = {}& 6 p_1\\
  (\omega_2) = {}& p_1+p_3+\sum_{j=1}^4 p_{2,j}\\
  ( \omega_3) = {}& 3p_1+3p_3\\
   (\omega_4) = {}& 6p_3 \ .
\end{align*}

\end{proof}

 In our case, the $\omega_i$ satisfy the easily verifiable equations

\begin{align*} 
   \omega_3^2 = {}& \omega_1 \omega_4 \\
   \omega_2^3 = {}& \omega_3( \omega_4^2- \omega_1^2) \ .
  \end{align*}

They represent I-WP as a complete intersection of a cubic and a quadric surface.

\begin{theorem}
The  Weierstrass data of I-WP can be given by 
the Gauss map $G=\omega_4/\omega_1$ and height differential $dh=\omega_3$ using 
\[
f(z)=\re \int^z \left(\frac1G-G, i\left(G+\frac1G \right), 2\right)\,dh
\]
\end{theorem}
\begin{proof}
In terms of the equation $y^{12}=x(x-1)^4$ the data become 
\begin{align*} 
   G = {}& x^{1/4} \\
   = {}& \frac{y^3}{x-1} \\
   dh= {}&  x^{4/12-1}(x-1)^{4/12-1}\, dx\\
   = {}&   \frac{y^4}{x(x-1)^{2}}\, dx\ .
  \end{align*}
The substitution $z=x^{1/4}$ leads to the representation of the surface as a branched cover over the octahedron:

\begin{align*} 
   G = {}& z\\
   dh= {}&   \frac{z}{(z(z^4-1))^{2/3}}\, dz\ .
\end{align*}
  
  We see that our $G$ is the degree 3 quotient map by the order 3 rotation to the sphere as needed, and our  $dh$ is invariant under the rotation about the $x_3$-axis, making it the height differential.

The López-Ros factor for $G$ is correct because we want $|G(1)|=1$, and the Bonnet factor for $dh$ is correct because we want $dh$ to be real on $(-\infty, 0)$ which represents a planar symmetry line of the surface.
\end{proof}

We will now prove our main result:

\begin{theorem}
In the associate family of I-WP, the surfaces corresponding to $\theta = k\cdot 60^\circ$ for $k\in\Z$ are congruent to I-WP, while the surfaces corresponding to $\theta = 30^\circ+k\cdot 60^\circ$ are congruent to Steßmann's surface.
\end{theorem}
\begin{proof}
Denote by $\phi(x,y) = (x,e^{2\pi i/3} y)$. This is an automorphism of $S_{147}$ of order 3 that acts on the Weierstrass representation by
\begin{align*} 
   G\circ \phi = {}& G\\
   \phi^*dh= {}&  e^{2\pi i/3} dh
\end{align*}
so that  in the associate family of  the  I-WP surface any two surfaces with Bonnet angles $\theta$ and $\theta+2\pi/3$ are isometric. In any associate family, the surfaces with Bonnet angles $\theta$ and $\theta+\pi$ are always isometric. The claims follow.
\end{proof}

\begin{figure}[h] 
   \centering
   \includegraphics[width=1.2in]{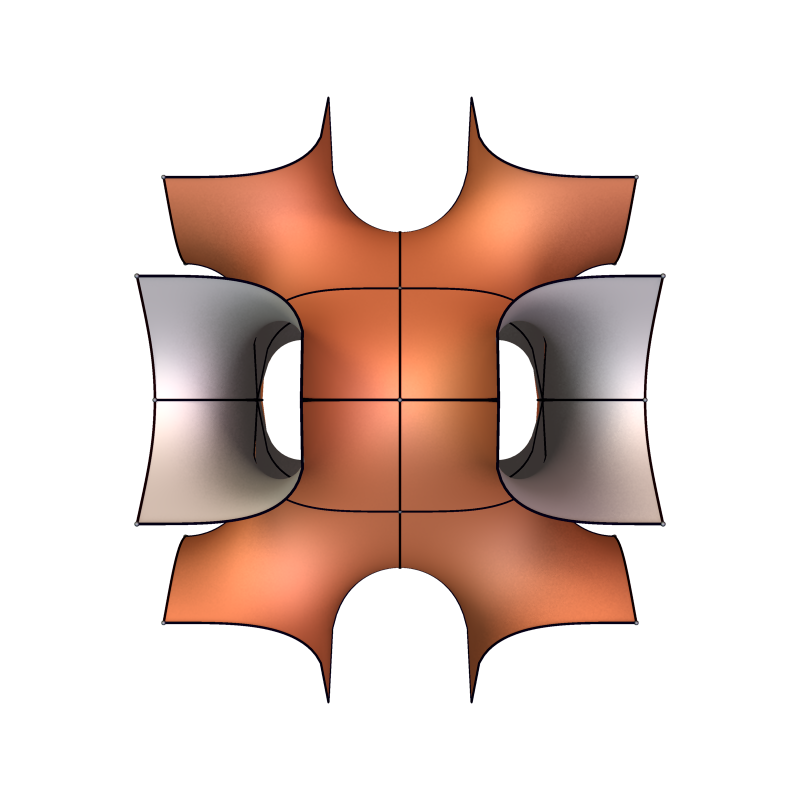} 
   \includegraphics[width=1.2in]{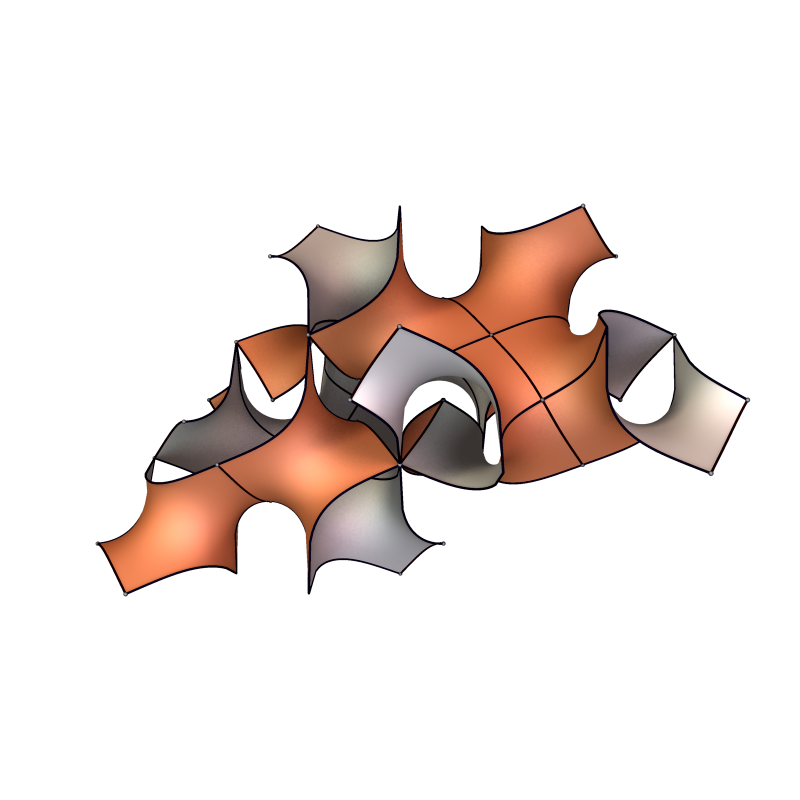} 
   \includegraphics[width=1.2in]{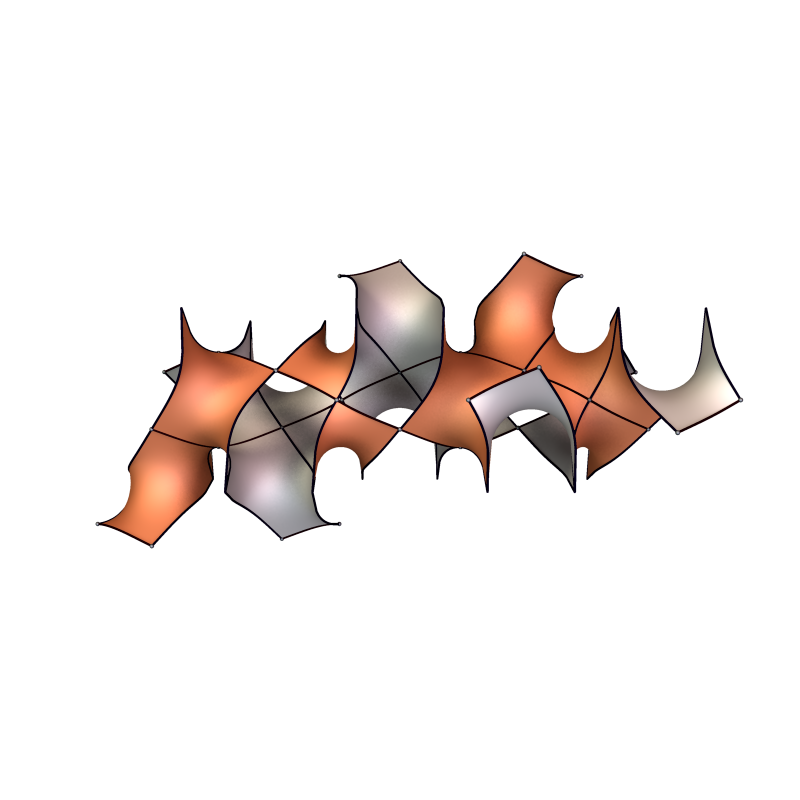} 
   \includegraphics[width=1.2in]{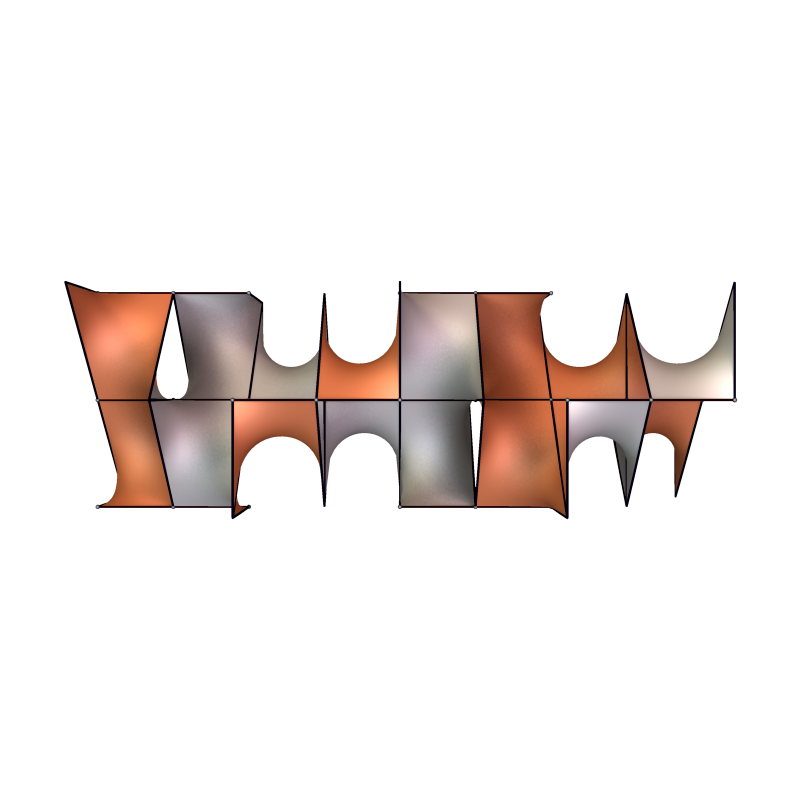} 
   \includegraphics[width=1.2in]{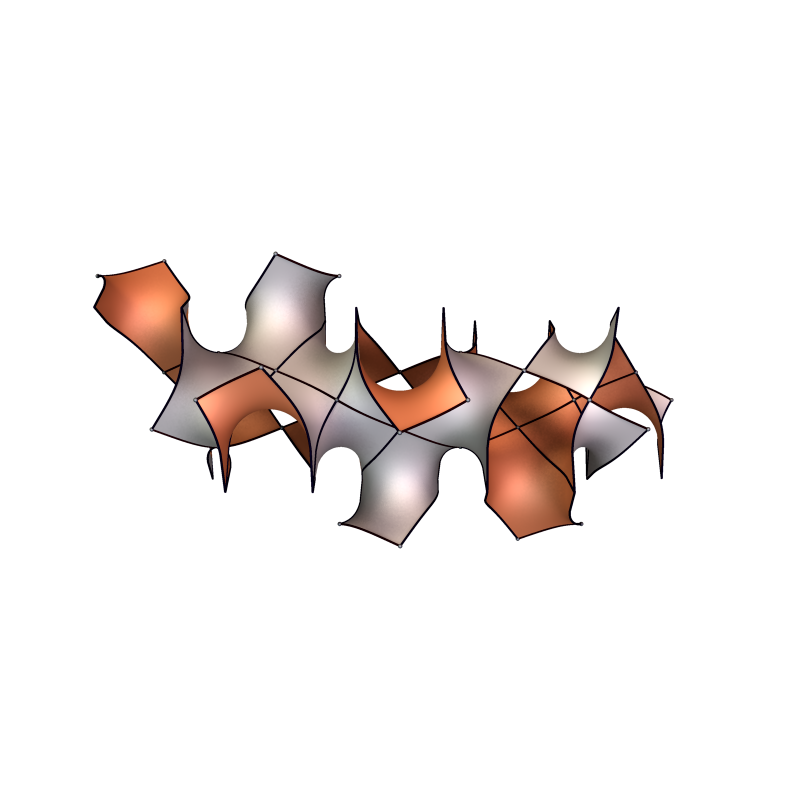} 
   \includegraphics[width=1.2in]{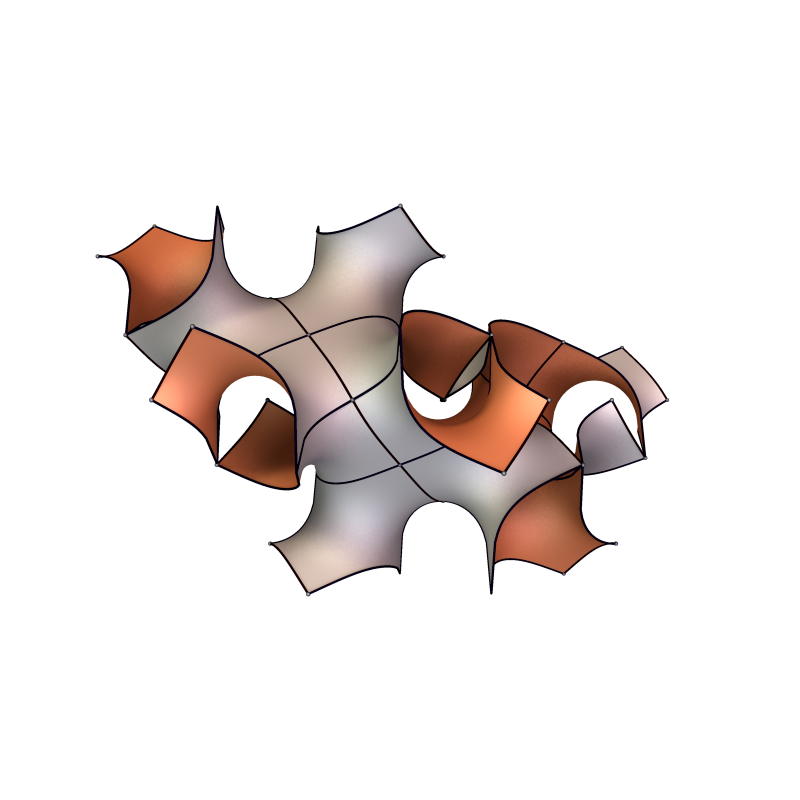} 
   \includegraphics[width=1.2in]{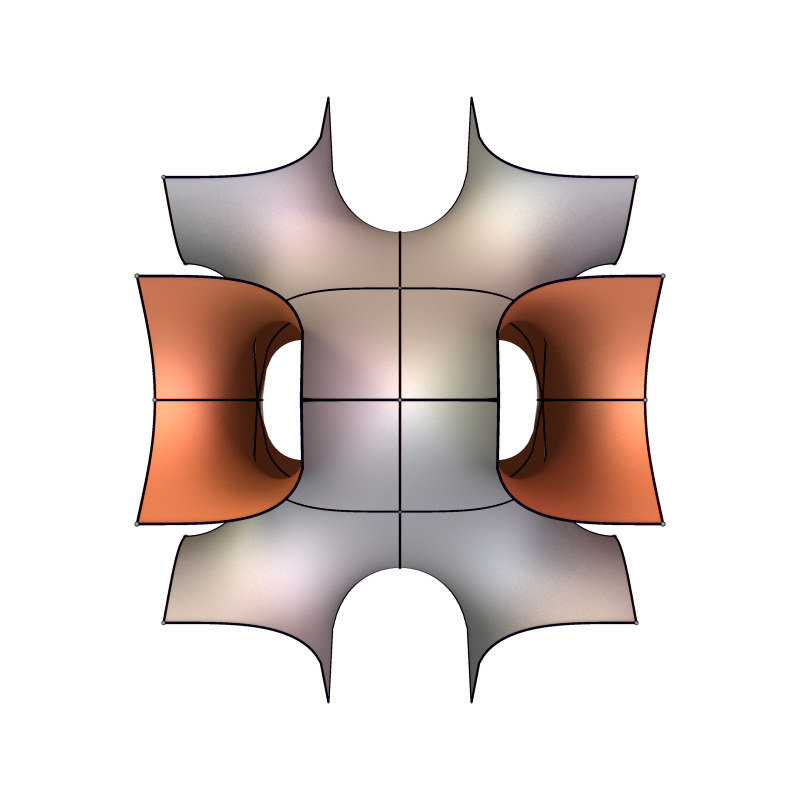} 
   \caption{The associate family of I-WP from $0^\circ$ to $60^\circ$}
   \label{fig:associate}
\end{figure}

\bibliography{minlit}
\bibliographystyle{alpha}

\end{document}